\documentclass[11pt]{article}
\usepackage[utf8]{inputenc}
\usepackage{amsmath} 
\usepackage{amsthm} 
\usepackage{amssymb}    
\usepackage{graphicx} 
\usepackage[dvips,letterpaper,margin=1in,bottom=0.7in]{geometry}
\usepackage{tensor}
\usepackage{physics} 
\usepackage{fancyhdr}
\usepackage{multicol}
\usepackage[inline]{enumitem}
\usepackage{comment}
\usepackage{float}
\usepackage{esint}
\usepackage{graphicx}
\usepackage[square,sort,comma,numbers]{natbib}\usepackage{afterpage}
\usepackage{pgf,tikz,pgfplots}
\usetikzlibrary{decorations.markings}
\usepackage{placeins}
\usepackage[bottom]{footmisc}
\usetikzlibrary{arrows}
\pgfplotsset{compat=1.14}
\definecolor{ffffff}{rgb}{1,1,1}
\definecolor{wrwrwr}{rgb}{0.3803921568627451,0.3803921568627451,0.3803921568627451}
\definecolor{aqaqaq}{rgb}{0.6274509803921569,0.6274509803921569,0.6274509803921569}
\definecolor{rvwvcq}{rgb}{0.08235294117647059,0.396078431372549,0.7529411764705882}

\newtheorem{thm}{Theorem}[section]
\newtheorem{prop}[thm]{Proposition}
\newtheorem{lem}[thm]{Lemma}
\newtheorem{cor}[thm]{Corollary}
\theoremstyle{definition}

\theoremstyle{remark}
\newtheorem*{remark}{Remark}
\numberwithin{equation}{section}

\newcommand{\Q}{\mathbb{Q}}

\newcommand{\Z}{\mathbb{Z}}

\newcommand{\N}{\mathbb{N}}
\renewcommand{\Re}{\operatorname{Re}}
\newcommand{\Mod}[1]{\ (\mathrm{mod}\ #1)}

\newcommand{\Addresses}{{
  \bigskip
  \footnotesize

  \noindent J.~Iskander, \textsc{2046 Deren Way NE, Atlanta, GA 30345}\par\nopagebreak
  \textit{Email address}: \texttt{jonasiskander@gmail.com}

  \medskip

  \noindent V.~Jain, \textsc{Department of Mathematics, Massachusetts Institute of Technology, Cambridge, MA 02139}\par\nopagebreak
  \textit{Email address}: \texttt{vanshika@mit.edu}

  \medskip

  \noindent V.~Talvola, \textsc{Department of Mathematics, Princeton University, Princeton, NJ 08544}\par\nopagebreak
  \textit{Email address}: \texttt{vtalvola@princeton.edu}

}}

\title{Exact Formulae for the Fractional Partition Functions}
\author{Jonas Iskander, Vanshika Jain, and Victoria Talvola}
\date{}

\begin{document}

\maketitle
\begin{abstract}
    The partition function $p(n)$ has been a testing ground for applications of analytic number theory to combinatorics. In particular, Hardy and Ramanujan invented the ``circle method'' to estimate the size of $p(n)$, which was later perfected by Rademacher who obtained an exact formula. Recently, Chan and Wang considered the fractional partition functions, defined for \(\alpha \in \Q\) by $\sum_{n = 0}^\infty p_{\alpha}(n)x^n := \prod_{k=1}^\infty (1-x^k)^{-\alpha}$. In this paper we use the Rademacher circle method to find an exact formula for \(p_\alpha(n)\) and study its implications, including log-concavity and the higher-order generalizations (i.e., the Turán inequalities) that \(p_\alpha(n)\) satisfies. 
\end{abstract}

\section{Introduction and Statement of Results}
A \textit{partition} of a nonnegative integer $n$ is a non-increasing sequence of positive integers with sum $n$. We use $p(n)$ to denote the number of partitions of $n$. One powerful tool for analyzing the partition function is Euler's generating function:
\begin{equation}
   P(x) := \sum_{n=0}^\infty p(n)x^n = \prod_{k=1}^\infty \frac{1}{1-x^k}.
\end{equation}
 
 \noindent
The study of the size of $p(n)$ spurred the development of the ``circle method,'' which has had many applications, including the proof of the weak Goldbach conjecture \cite{helfgott}. In 1918, G. H. Hardy and S. Ramanujan \cite{hardy} invented this method to obtain an infinite but divergent series expansion for \(p(n)\) and the asymptotic formula: $$p(n) \sim \frac{e^{\pi \sqrt{2n/3}}}{4n\sqrt{3}}.$$ This method was perfected by H. Rademacher \cite{rademacher}, who determined the convergent exact formula \begin{equation}\label{rseries}
    p(n) = \frac{2 \pi}{(24n-1)^\frac{3}{4}} \sum_{k=1}^{\infty} \frac{A_k(n)}{k} \cdot I_{\frac{3}{2}}\left(\frac{\pi}{6k} \sqrt{24n-1}\right),
\end{equation} where \begin{equation*}
    I_{\nu}(z) := \left(\frac{z}{2}\right)^\nu \sum_{k=0}^\infty \frac{(z/2)^{2k}}{k!\Gamma(\nu+k+1)}
\end{equation*} is the modified Bessel function of the first kind, \begin{equation*}
    A_k(n) := \!\!\!\!\sum_{\substack{0 \leq h < k \\ \gcd(h, k) = 1}}\!\!\!\! e^{\pi i s(h, k) - 2 \pi i n h / k}
\end{equation*} is a Kloosterman sum, and \begin{equation}\label{dedekindsum}
    s(h,k) := \sum_{r = 1}^{k - 1} \frac{r}{k} \left(\frac{hr}{k}-\left\lfloor\frac{hr}{k}\right\rfloor-\frac{1}{2}\right)
\end{equation} is the usual Dedekind sum.


The partition function also satisfies certain congruences, which exhibit a great degree of structure. Ramanujan was the first to study these congruences, and he discovered examples including \(p(5n+4) \equiv 0 \pmod{5}\). In a recent paper, Chan and Wang \cite{chan} defined for $\alpha \in \Q$ the \textit{fractional partition function} \(p_\alpha(n)\) in terms of its generating function \begin{equation}\label{palphadef}
    P(x)^\alpha = \prod_{k = 1}^\infty \frac{1}{(1-x^k)^\alpha} =: \sum_{n=0}^\infty p_\alpha(n)x^n,
\end{equation} and studied its congruences, showing, for instance, that \(p_{1/2}(29n + 26) \equiv 0 \pmod{29}\). A general theory of such congruences has recently been developed by Bevilacqua, Chandran, and Choi \cite{bevilacqua}. The discussion of congruences for \(p_\alpha(n)\) is possible because \(p_\alpha(n)\) is rational whenever \(\alpha\) is rational.

When $\alpha \in \Z^+$, $p_\alpha(n)$ counts the number of partitions of $n$ in which each term is labeled with one of $\alpha$ different colors, where the order of the colors does not matter \cite{keith}. Moreover, in such cases, the function \begin{align}
    \eta(\tau)^{-\alpha} & = q^{-\frac{\alpha}{24}} P(q)^{\alpha}
\end{align} is a weakly holomorphic modular form of weight $-\alpha/2 \in (1/2)\Z$, where $\tau$ is in the upper half-plane, $\eta(\tau) := q^{1/24} \prod_{n \geq 1} (1 - q^n)$ is the Dedekind eta function, and $q := e^{2 \pi i \tau}$. This makes it possible to compute the values of $p_\alpha(n)$ using Maass-Poincar\'e series, as described by Bringmann et al. \cite[\S6.3]{bringmann}, which give a Rademacher-type infinite series expansion that reduces to (\ref{rseries}) when $\alpha = 1$. To do this, one computes the principal part of $\eta(\tau)^{-\alpha}$, which correspond to the values \(p_\alpha(n)\) for \(0 \leq n \leq \lfloor\alpha/24\rfloor\). Then, using the fact that a weakly holomorphic modular form is determined by its weight and principal part, one can write it as a finite sum of Maass-Poincar\'e series and apply a known formula for the coefficients of such series. While these observations shed light on the case where \(\alpha\) is a positive integer, there is currently no known combinatorial or modular-form interpretation of \(p_\alpha(n)\) for arbitrary rational \(\alpha\).

In this paper, we extend the definition of $p_\alpha(n)$ to arbitrary real \(\alpha\) via (\ref{palphadef}) and give exact formulas for \(p_{\alpha}(n)\) in the spirit of Rademacher. For real $\alpha > 0$, $n > \alpha/24$, and \(m \leq \alpha/24\), we define the functions \begin{equation}\label{mu}
    \nu_\alpha(n) := \sqrt{n - \frac{\alpha}{24}}, \quad \mu_\alpha(m) := \sqrt{\frac{\alpha}{24} - m}
\end{equation} and the \textit{$\alpha$-Kloosterman sum} \begin{equation}
A_k^{(\alpha)}(n, m) := \sum_{\substack{0 \leq h < k \\ (h, k) = 1}} e^{\alpha \pi i s(h, k) + \frac{2 \pi i}{k}(mH - nh)},
\end{equation} where \(H\) denotes an inverse of \(h\) modulo \(k\) and $s(h, k)$ is the Dedekind sum defined in (\ref{dedekindsum}). Our exact formulas for $p_\alpha(n)$ are the content of the following theorem.

\begin{thm}\label{main theorem}
For all \(\alpha > 0\) and \(n > \alpha/24\), we have \begin{align} \label{main series}
     p_\alpha(n) &= \nu_\alpha(n)^{-\frac{\alpha}{2}-1}\sum_{m=0}^q \mu_\alpha(m)^{\frac{\alpha}{2}+1} p_\alpha(m) \sum_{k = 1}^\infty \frac{2\pi}{k} A_k^{(\alpha)}(n, m)I_{\frac{\alpha}{2}+1}\left(\frac{4\pi}{k}\nu_\alpha(n)\mu_\alpha(m)\right),
\end{align} where \(q := \lfloor \frac{\alpha}{24} \rfloor\).

\end{thm}

\noindent
Theorem~\ref{main theorem} also enables the calculation of explicit error bounds for approximations of $p_\alpha(n)$ obtained by truncating \eqref{main series}. These have several implications, including a simple description of the asymptotic behavior of $p_\alpha(n)$ for large $n$, given in Corollary~\ref{asymp}.

\begin{cor}\label{asymp}
For all \(\alpha > 0\), as \(n \to \infty\), we have \begin{align*}
    p_\alpha(n) &\sim 2\pi\frac{I_{\frac{\alpha}{2}+1}\left(\frac{\pi\alpha}{6}\lambda_\alpha(n)\right)}{\lambda_\alpha(n)^{\frac{\alpha}{2}+1}} \sim \sqrt{\frac{12}{\alpha}} \cdot \frac{ e^{\frac{\alpha\pi}{6}\lambda_\alpha(n)}}{\lambda_\alpha(n)^{\frac{\alpha + 3}{2}}},
\end{align*} where \(\lambda_\alpha(n) := \sqrt{\frac{24n}{\alpha}-1}\).
\end{cor}

\noindent
We remark that because \(p_\alpha(n)\) is rational for any \(\alpha \in \Q\), Theorem \ref{main theorem} implies that the series in \eqref{main series} converges to a rational number when \(n \in \Z\), \(n > \alpha/24\). We make use of this fact later in the paper (Corollary \ref{rational}) to provide a finite formula for \(p_\alpha(n)\) in the case where \(\alpha \in \Q\).

When considering a sequence of real numbers, one is often interested in more than just its asymptotic behavior. One property that is often studied is log-concavity. A sequence \(\{a(n)\}\) is called \textit{log-concave} if we have \[a(n+1)^2 - a(n)a(n+2) \geq 0\] for all \(n\). Nicolas \cite{nicolas} and DeSalvo and Pak \cite{DeSalvo} independently proved that \(p(n)\) is log-concave for \(n \geq 25\). In fact, the condition of log-concavity is a special case of what are known as the \textit{higher Turán inequalities} \cite{chen}. One can show that a sequence satisfies the higher Turán inequalities of degree \(d\) if and only if the Jensen polynomials \begin{align}
    J_{a}^{d, n}(x) := \sum_{j = 0}^d \binom{d}{j} a(n + j) x^j
\end{align} have strictly real roots for all \(n\)---we say that such a polynomial is \textit{hyperbolic} \cite{craven1989jensen}. Chen, Jia, and Wang \cite{chen} conjectured that for any fixed degree \(d\), \(J_p^{d,n}(x)\) is eventually hyperbolic, and proved this for $d = 3$; Larson and Wagner \cite{larson} independently proved this conjecture for $d \in \{3, 4, 5\}$. Griffin, Ono, Rolen, and Zagier \cite{griffin} established the conjecture of Chen et al.\@ for all $d$ by showing that, after suitable renormalization, the Jensen polynomials of \(p(n)\) converge to the Hermite polynomials \(H_d(x)\) as \(n \to \infty\). We apply their methods to prove the analogue of Chen et al.'s conjecture for $p_\alpha(n)$.

\begin{thm}\label{jensen}
For $\alpha > 0$ and $d \in \N$, there exists $N_d(\alpha)$ such that $J_{p_\alpha}^{d, n}(X)$ is hyperbolic for all $n > N_d(\alpha)$.
\end{thm}

Our paper is divided into five main sections. In Section 2, we establish some preliminary results, including a modified version of the Dedekind functional equation for $\eta(\tau)$. In Section 3, we use the circle method along with this identity to prove Theorem~\ref{main theorem}. In Section 4, we use Theorem~\ref{main theorem} to prove more results about $p_\alpha(n)$, including the estimate given in Corollary~\ref{asymp}. We also analyze the hyperbolicity of the Jensen polynomials associated with \(p_\alpha(n)\). Finally, in Section 5 we provide numerical illustrations of our main theorems.

\subsection*{Acknowledgements}
The authors would like to thank Ken Ono, Larry Rolen, and Ian Wagner for suggesting the problem and their guidance. The research was supported by the generosity of the Asa Griggs Candler Fund, the National Security Agency under grant H98230-19-1-0013, and the National Science Foundation under grants 1557960 and 1849959.   

\section{Proof of the Functional Equation for $P(x)^{\alpha}$}

In order to apply the circle method to \(p_\alpha(n)\), we first require a precise statement of Dedekind's functional equation for the eta function. We derive this from Iseki's formula \cite[\S3.5]{apostol}. For convenience, when $\Re(x)>0$, we set \begin{equation*}
    \lambda(x) := \sum_{m=1}^\infty \frac{e^{-2\pi mx}}{m} = -\log(1-e^{-2\pi x}).
\end{equation*}

\begin{remark}
Throughout this section, we let \(\log{z}\) denote the branch of the logarithm with a branch cut along the negative imaginary axis and \(\log{1} = 0\), and we define \(\arg{z} := \Im(\log{z})\). 
\end{remark}

\subsection{Derivation of the Logarithmic Functional Equation from Iseki's Formula}

In order to derive the required modification of the functional equation for \(\eta(\tau)\), we first prove a lemma which follows from Iseki's formula \cite[\S3.5]{apostol}.

\begin{thm}[Iseki's Formula]
For \(\Re z > 0\), \(0 < \alpha < 1\), and \(0 \leq \beta \leq 1\), let \begin{equation*}
    \Lambda(\alpha, \beta, z) := \sum_{r=0}^\infty\left[\lambda((r+\alpha)z-i\beta)+\lambda((r+1-\alpha)z+i\beta)\right].
\end{equation*} Then we have \begin{equation}
    \Lambda(\alpha, \beta, z) = \Lambda(1-\beta, \alpha, z^{-1}) - \pi z\left(\alpha^2 - \alpha + \frac{1}{6}\right) + \frac{\pi}{z}\left(\beta^2 - \beta + \frac{1}{6}\right) + 2\pi i\left(\alpha - \frac{1}{2}\right)\left(\beta - \frac{1}{2}\right).
\end{equation}
\end{thm}

\begin{lem}\label{logfunlem}
For \(\Re z > 0\), we have \begin{equation}\label{isekilem}
     \sum_{r=1}^\infty \lambda(rz) = \sum_{r=1}^\infty \lambda \left(\frac{r}{z} \right) + \frac{1}{2}\log{z} - \left(\frac{\pi z}{12} - \frac{\pi}{12z}\right).
\end{equation}
\end{lem}

\begin{proof}
Letting \(\beta = 0\) in Iseki's formula, we obtain \begin{equation}\label{iseki0}
    \Lambda(\alpha, 0, z) = \Lambda(1, \alpha, z^{-1}) - \pi z\left(\alpha^2 - \alpha + \frac{1}{6}\right) + \frac{\pi}{6z} - \pi i\left(\alpha - \frac{1}{2}\right).
\end{equation} From here, bringing \(\Lambda(1, \alpha, z^{-1})\) to the left side, reordering the summations, and setting \(a(\alpha) := \lambda(\alpha z) - \lambda(i\alpha)\) and \begin{equation*}
    b_r(\alpha) := \lambda((r+\alpha)z)+\lambda((r-\alpha)z) - \lambda\left(\frac{r}{z}-i\alpha\right) - \lambda\left(\frac{r}{z}+i\alpha\right)
\end{equation*} yields \begin{equation}\label{isekiprelim}
    a(\alpha) + \sum_{r=1}^\infty b_r(\alpha) = - \pi z\left(\alpha^2 - \alpha + \frac{1}{6}\right) + \frac{\pi}{6z} - \pi i\left(\alpha - \frac{1}{2}\right).
\end{equation} The reordering is valid because the sum over each of the four terms in \(b_r(\alpha)\) converges absolutely, since \(\lambda(\gamma z) \sim e^{-2\pi\gamma z}\) as \(\gamma \to \infty\). We proceed by taking the limit as \(\alpha \to 0^+\). We start by observing that \begin{align}
    \lim_{\alpha \to 0^+} a(\alpha) &= \lim_{\alpha \to 0^+}[\lambda(\alpha z) - \lambda(i\alpha)] \nonumber \\
    &= \lim_{\alpha \to 0^+}\left[\log(1-e^{-2\pi i\alpha}) - \log(1-e^{-2\alpha\pi z})\right] \nonumber \\
    &= \lim_{\alpha \to 0^+}\log\left(\frac{1-e^{-2\pi i\alpha}}{1-e^{-2\alpha\pi z}}\right),
\end{align} where the last step is justified because \(\arg(1-e^{-2\pi i\alpha}) - \arg(1-e^{-2\alpha\pi z}) \in (-\pi, \pi)\) for \(\alpha > 0\). By L'Hôpital's rule, \begin{align*}
    \lim_{\alpha \to 0^+}\frac{1-e^{-2\pi i\alpha}}{1-e^{-2\alpha\pi z}} = \lim_{\alpha \to 0^+}\frac{2\pi ie^{-2\pi i\alpha}}{2\pi ze^{-2\alpha\pi z}} = \frac{i}{z},
\end{align*} and so \begin{align*}
    \lim_{\alpha \to 0^+} a(\alpha) = \log\left(\frac{i}{z}\right) = \frac{\pi i}{2} - \log{z},
\end{align*} using the fact that \(\arg(i/z) \in (0, \pi)\) and \(\log(1/z) = -\log{z}\) for our definition of the logarithm. We now show that \begin{equation*}
    \lim_{\alpha \to 0^+} \sum_{r=1}^\infty b_r(\alpha) = \sum_{r=1}^\infty \lim_{\alpha \to 0^+} b_r(\alpha) = \sum_{r=1}^\infty(2\lambda(rz) - 2\lambda(r/z)).
\end{equation*} For this purpose, start by noting that for \(\Re x > 0\), we have \(|\lambda(x)| \leq \lambda(\Re x)\) by the series expansion for \(\lambda\), and that \(\lambda(\Re x)\) is monotonically decreasing. In particular, we have \begin{equation*}
    |\lambda((r\pm\alpha)z)| \leq \lambda\left(\left(r-\frac{1}{2}\right)\Re z\right) \leq \lambda\left(r \cdot \frac{\Re z}{2}\right),
\end{equation*} and \(|\lambda(rz\pm i\alpha)| \leq \lambda(r\Re z).\) For \(x > 0\), we can verify that \(\sum_{r=1}^\infty \lambda(rx)\) converges by the asymptotic behavior of \(\lambda(rx)\) as \(r \to \infty\). Consequently, by the discrete version of the dominated convergence theorem, we may exchange the order of the limit and the summation over \(b_r\). Thus, in the limit, (\ref{isekiprelim}) becomes \begin{equation}
    \frac{\pi i}{2} - \log{z} + 2\sum_{r=1}^\infty \lambda(rz) - 2\sum_{r=1}^\infty \lambda\left(\frac{r}{z}\right) = -\frac{\pi z}{6} + \frac{\pi}{6z} + \frac{\pi i}{2}.
\end{equation} This is equivalent to (\ref{isekilem}).
\end{proof}

For the main theorem of this section, we begin by citing a fact proven in \cite[\S3.6]{apostol}.

\begin{prop}
Let \(\Re{z} > 0\), let \(h, k \in \Z\) be coprime with \(k > 0\), and choose $H$ such that \(hH \equiv -1 \pmod{k}\). Then we have that \begin{equation}\label{apostolprop}
    \sum_{\substack{n = 1 \\ n \not\equiv 0 \Mod{k}}}^\infty \hspace{-0.4cm} \lambda\left(\frac{n}{k}(z - ih)\right) = \hspace{-0.4cm} \sum_{\substack{n = 1 \\ n \not\equiv 0 \Mod{k}}}^\infty \hspace{-0.4cm} \lambda\left(\frac{n}{k}(z^{-1} - iH)\right) + \left(\frac{\pi z}{12} - \frac{\pi}{12z}\right)\left(1 - \frac{1}{k}\right) + \pi is(h, k).
\end{equation}
\end{prop}

\noindent
With this fact and Lemma~\ref{logfunlem}, we may finally provide the desired logarithmic version of the functional equation for \(\eta(\tau)\).

\begin{thm} \label{logfun}
For \(\Re{z} > 0\) and \(h, k, H \in \Z\) with \(k > 0\), \(\gcd(h,k) = 1\), and \(hH \equiv -1 \pmod{k}\), we have \begin{equation}
    \sum_{n=1}^\infty \lambda\left(\frac{n}{k}(z - ih)\right) = \sum_{n=1}^\infty \lambda\left(\frac{n}{k}(z^{-1} - iH)\right) + \frac{1}{k}\left(\frac{\pi}{12z} - \frac{\pi z}{12}\right) + \frac{1}{2}\log{z} + \pi is(h, k).
\end{equation}
\end{thm}

\begin{proof}
Using the periodicity of \(\lambda\), we note that \begin{equation*}
    \sum_{r=1}^\infty \lambda(rz) = \hspace{-0.4cm} \sum_{\substack{n=1 \\ n \equiv 0 \Mod{k}}}^\infty \hspace{-0.4cm}\lambda\left(\frac{n}{k}(z - ih)\right) \quad \text{and} \quad \sum_{r=1}^\infty \lambda\left(\frac{r}{z}\right) = \hspace{-0.4cm} \sum_{\substack{n=1 \\ n \equiv 0 \Mod{k}}}^\infty \hspace{-0.4cm} \lambda\left(\frac{n}{k}(z^{-1} - iH)\right).
\end{equation*} Substituting this into (\ref{isekilem}) and adding equation (\ref{apostolprop}) yields the desired result.
\end{proof}

\subsection{Application of the Logarithmic Functional Equation to \(P(x)^\alpha\)}


We recall the generating function \begin{equation*}
    P(x) := \prod_{k=1}^\infty \frac{1}{1-x^k} = \sum_{n=0}^\infty p(n)x^n,
\end{equation*} which is holomorphic for \(x\) in the open unit disk. In deriving the Hardy-Ramanujan-Rademacher series formula for the partition function, we rely on the fact that the equation above holds analytically as well as formally. We extend this observation to \(p_\alpha(n)\) by showing that the generating function \(P(x)^\alpha\) is well-defined.

\begin{lem} \label{pxalphagen}
For \(x\) in the open unit disk and \(\alpha>0\), we have \begin{equation}
    \sum_{n=0}^\infty p_\alpha(n)x^n = \prod_{k=1}^\infty e^{-\alpha\log(1 - x^k)} =: P(x)^\alpha.
\end{equation}
\end{lem}

\begin{proof}
Start by observing that our branch of the logarithm ensures that \(\exp(-\alpha\log(1-x^k))\) is formally equivalent to \((1-x^k)^{-\alpha}\). Thus, because the \(p_\alpha(n)\) are defined in terms of the formal equivalence in (\ref{palphadef}), it suffices to show that \(P(x)^\alpha\) as defined above is holomorphic for \(|x| < 1\). For this purpose, let \(0 < r < 1\), and observe that \(-\sum_{k=1}^\infty \alpha\log(1-x^k)\) converges uniformly for \(|x| \leq r\) by the ratio test, as \begin{equation*}
    \lim_{k \to \infty} \left|\frac{\log(1-x^{k+1})}{\log(1-x^k)}\right| = \lim_{k \to \infty} \left|\frac{-\log(x)x^{k+1}/(1-x^{k+1})}{-\log(x)x^k/(1-x^k)}\right| = \lim_{k \to \infty} \left|x \cdot \frac{1-x^k}{1-x^{k+1}}\right| = |x| \leq r.
\end{equation*}

\noindent
Thus, \(\prod_{k=1}^\infty e^{-\alpha\log(1-x^k)}\) converges uniformly for \(|x| \leq r\), from which it follows that \(P(x)^\alpha\) is holomorphic in every closed disk \(|x| \leq r\) and hence in the open unit disk \(|x| < 1\) as desired.

\end{proof}

\noindent
We are finally ready for the main result of this section, which expresses the functional equation for \(\eta(\tau)\) in terms of \(P(x)^\alpha\).

\begin{thm}[Modified Functional Equation]\label{functional equation}
For \(\Re{z} > 0\), \(\alpha>0\), \(h, k, H \in \Z\) with \(k > 0\), \(\gcd(h, k) = 1\), and \(hH \equiv -1 \pmod{k}\), we have \begin{equation}\label{functional}
    P(x)^\alpha = e^{\pi i\alpha s(h, k)}\left(\frac{z}{k}\right)^{\alpha/2}\exp\left(\frac{\alpha\pi}{12k}\left(\frac{k}{z} - \frac{z}{k}\right)\right)P(x')^\alpha,
\end{equation} where \begin{align} \label{xprime}
    x & := \exp\left(\frac{2\pi}{k}\left(ih - \frac{z}{k}\right)\right), \qquad x' := \exp\left(\frac{2\pi}{k}\left(iH - \frac{k}{z}\right)\right),
\end{align} and real powers are given for the precise branch of the logarithm described in Section 2.1.
\end{thm}

\begin{proof}
Applying Theorem~\ref{logfun} with \(z/k\) in place of \(z\) and multiplying by \(\alpha\), we obtain \begin{equation*}
    \alpha\sum_{n=1}^\infty \lambda\left(\frac{n}{k}\left(\frac{z}{k} - ih\right)\right) =  \frac{\alpha\pi}{12k}\left(\frac{k}{z} - \frac{z}{k}\right) + \frac{\alpha}{2}\log\left(\frac{z}{k}\right) + \pi i\alpha s(h, k) +  \alpha\sum_{n=1}^\infty \lambda\left(\frac{n}{k}\left(\frac{k}{z} - iH\right)\right).
\end{equation*} Exponentiating both sides yields \begin{equation}
    \prod_{n=1}^\infty \exp\left(-\alpha \log(x)\right) = \exp\left(\frac{\alpha\pi}{12k}\left(\frac{k}{z} - \frac{z}{k}\right)\right)\exp\left(\frac{\alpha}{2}\log\left(\frac{z}{k}\right)\right)e^{\pi i\alpha s(h,k)}\prod_{n=1}^\infty \exp\left(-\alpha \log(x')\right)
\end{equation} for \(x\) and \(x'\) defined above, which is equivalent to (\ref{functional}).
\end{proof}

\section{Proof of the Series Formula for \(p_\alpha(n)\)}

In this section, we use Radamacher's circle method to prove the series formula for \(p_\alpha(n)\). We closely follow Apostol's proof of the $\alpha = 1$ case \cite[\S5.7]{apostol}.

\begin{proof}[Proof of Theorem~\ref{main theorem}]

Using Cauchy's residue theorem and Lemma~\ref{pxalphagen}, we can write \begin{align} \label{3.1} p_\alpha(n) = \frac{1}{2 \pi i} \int_C \frac{P(x)^{\alpha}}{x^{n+1}} \, dx,\end{align} where $C$ is any simple closed contour in the unit disk which encloses the origin. To evaluate this, we consider the change of variables $x = e^{2 \pi i \tau}$, under which the closed unit disk $|x| \leq 1$ is the image of the infinite vertical strip \(\{\tau : 0 \leq \Re \tau \leq 1, \,\, 0 \leq \Im \tau\}\). We start by recalling the Farey sequences \(F_N\), defined by enumerating the rational numbers in \([0, 1]\) with reduced denominators at most \(N\). In addition, for \(\gcd(h, k) = 1\), we let \(C(h, k)\) denote the Ford circle associated with \(h/k\), which has center \(h/k + i/(2k^2)\) and radius \(1/(2k^2)\) (details are given in \cite[\S5.6]{apostol}). As in Rademacher's original work, we integrate along the Rademacher paths $R(N)$ in the \(\tau\)-plane, consisting of the upper arcs of the Ford circles associated with $F_N$, with the intent to later take the limit as \(N \to \infty\) (depicted in Figure~\ref{r2}). For $N \geq 1$, we write (\ref{3.1}) as \begin{align} \label{palphaint}
    p_{\alpha}(n) = \int_{R(N)} P(e^{2 \pi i \tau})^{\alpha} e^{-2 \pi i n \tau} d\tau.
\end{align}
Decomposing \(R(N)\) into its component arcs, we may write the above integral as \begin{align} \int_{R(N)} = \sum_{k=1}^N \sum\limits_{\substack{0 \leq h < k \\ (h, k) = 1}} \int_{\gamma(h, k)} =: \sum_{h, k} \int_{\gamma(h, k)}, \end{align} where we define the right side as a shorthand for the double sum over \(h\) and \(k\), and $\gamma(h, k)$ is the upper arc of the Ford circle $C(h, k)$ of radius \(1/(2k^2)\) tangent to the real axis at \(h/k\). 

\begin{figure}
\begin{minipage}{.5\linewidth}
    \vspace{5pt}
    \centering
    \begin{tikzpicture}[line cap=round,line join=round,>=triangle 45,x=1cm,y=1cm, thick,scale=0.25]
\draw [line width=2pt,color=aqaqaq] (-16.623008274372093,-1.4705519580355455)-- (11.188001819156527,-1.3616932499059609);
\draw [shift={(-16.677437628436884,12.434953088728765)},line width=2pt,color=aqaqaq]  plot[domain=-1.5668821161344475:1.574710537455346,variable=\t]({1*13.905611570879278*cos(\t r)+0*13.905611570879278*sin(\t r)},{0*13.905611570879278*cos(\t r)+1*13.905611570879278*sin(\t r)});
\draw [shift={(11.133572465091735,12.54381179685835)},line width=2pt,color=aqaqaq]  plot[domain=1.5747105374553463:4.71630319104514,variable=\t]({1*13.905611570879278*cos(\t r)+0*13.905611570879278*sin(\t r)},{0*13.905611570879278*cos(\t r)+1*13.905611570879278*sin(\t r)});
\draw [line width=2pt,color=aqaqaq] (-2.7313021420734995,2.109197107052741) circle (3.447524820039072cm);
\draw [shift={(-2.7313021420734995,2.109197107052741)},line width=2pt] plot[domain=0.6595265983587923:2.489894476551889,variable=\t]({1*3.447524820039073*cos(\t r)+0*3.447524820039073*sin(\t r)},{0*3.447524820039073*cos(\t r)+1*3.447524820039073*sin(\t r)}) ;
\draw [line width=2pt,color=ffffff] (-5.472273772434161,4.200249929518844)-- (-2.7313021420734995,2.109197107052741);
\draw [line width=2pt,color=ffffff] (-0.006783975682795873,4.22164311720128)-- (-2.7313021420734995,2.109197107052741);
\draw [shift={(-16.677437628436884,12.434953088728765)},line width=2pt]  plot[domain=-0.6337684967357191:1.5747105374553463,variable=\t]({1*13.905611570879278*cos(\t r)+0*13.905611570879278*sin(\t r)},{0*13.905611570879278*cos(\t r)+1*13.905611570879278*sin(\t r)}) ;
\draw [ shift={(11.133572465091735,12.54381179685835)},line width=2pt]  plot[domain=1.5747105374553463:3.783189571646414,variable=\t]({1*13.905611570879278*cos(\t r)+0*13.905611570879278*sin(\t r)},{0*13.905611570879278*cos(\t r)+1*13.905611570879278*sin(\t r)}) ;
\draw [line width=2pt,color=ffffff] (-16.73186698250168,26.340458135493076)-- (-16.677437628436884,12.434953088728765);
\draw [line width=2pt,color=ffffff] (-16.677437628436884,12.434953088728765)-- (-5.472273772434161,4.200249929518844);
\draw [line width=2pt,color=ffffff] (11.133572465091735,12.54381179685835)-- (-0.006783975682795873,4.22164311720128);
\draw [line width=2pt,color=ffffff] (11.133572465091735,12.54381179685835)-- (11.079143111026942,26.44931684362266);
\begin{scriptsize}
\draw [fill=black] (-16.623008274372093,-1.4705519580355455) circle (2.5pt);
\draw[color=black] (-16.346709728056492,-3) node {\large $0$};
\draw [fill=black] (11.188001819156527,-1.3616932499059609) circle (2.5pt);
\draw[color=black] (11.64241862931025,-3) node {\large $1$};
\draw [fill=wrwrwr] (-16.73186698250168,26.340458135493076) circle (2pt);
\draw[color=black] (-17.933138056004406,26.88804669565975) node {\large $i$};
\draw [fill=wrwrwr] (11.079143111026942,26.44931684362266) circle (2pt);
\draw[color=black] (14.701958976066937,26.68126085963371) node {\large $i + 1$};
\draw [fill=black] (-2.8477417939116583,-1.4166323877588578) circle (2.5pt);
\draw[color=black] (-2.5, -3.3) node {\large $\frac{1}{2}$};
\end{scriptsize}
\end{tikzpicture}
\caption{The Rademacher path $R(2)$}
\end{minipage}
\begin{minipage}{.5\linewidth}
    \centering
    \label{zplane}
    \vspace{5pt}
    \begin{tikzpicture}[line cap=round,line join=round,x=1cm,y=1cm, thick,scale=0.565]
    \clip(-2.9,-8.777292400217647) rectangle (10.9,5.673001722083454);
\draw [line width=2pt,color=gray] (4.26,-0.58) circle (4.3cm);
\draw [line width=2pt,color=gray,domain=-5:13] plot(\x,{(-2.494-0*\x)/4.3});
\draw [line width=2pt,color=wrwrwr] (2.085918052888333,3.1299013042455917)-- (8.118080161419881,-2.478741021850621);
\draw [decoration={markings, mark=at position 0.3 with {\arrow{<}}, mark=at position 0.79 with {\arrow{>}}},
        postaction={decorate}, shift={(4.26,-0.58)},line width=2pt,color=black]  plot[domain=-0.4573451714395125:2.1008741114393823,variable=\t]({1*4.3*cos(\t r)+0*4.3*sin(\t r)},{0*4.3*cos(\t r)+1*4.3*sin(\t r)}) -- cycle ;
\begin{scriptsize}
\draw [fill=rvwvcq] (-0.04,-0.58) circle (2.5pt);
\draw[color=black] (0.3989113036340631,0.018158545360415168) node {\large$0$};
\draw [fill=rvwvcq] (4.26,-0.58) circle (2.5pt);
\draw[color=black] (4.26,-1.4885703605800993) node {\large $\frac{1}{2}$};
\draw [fill=wrwrwr,] (8.56,-0.58) circle (2pt);
\draw[color=black] (9.000963603003536,-0.03663159667378535) node {\large $1$};
\draw [fill=rvwvcq] (2.085918052888333,3.1299013042455917) circle (2.5pt);
\draw[color=black] (2.124800777711378,4.1548142689425545) node {\large $z_1(h, k)$};
\draw[color=black] (0.124800777711378,2.1548142689425545) node {\large $K$};
\draw [fill=rvwvcq] (8.118080161419881,-2.478741021850621) circle (2.5pt);
\draw[color=black] (9.713235449448142,-2.5295830592299096) node {\large $z_2(h, k)$};
\end{scriptsize}
\end{tikzpicture}
\caption{Path of integration in the $z$-plane}\label{r2}
\end{minipage}

\end{figure}

We now introduce a second change of variables given by \begin{align} z & = -ik^2 \left(\tau - \frac{h}{k}\right), \end{align} which maps the circle \(C(h, k)\) onto the circle $K$ of radius \(1/2\) centered at \(1/2\). Let \(z_1(h, k)\) and \(z_2(h,k)\) be the respective endpoints of the image of \(\gamma(h, k)\), and let \(x\) and \(x'\) be defined as in Theorem~\ref{functional equation}. Then \begin{align*} p_{\alpha}(n) & = \sum_{h,k} i k^{-2} e^{-\frac{2 \pi i n h}{k}} \int_{z_1(h, k)}^{z_2(h, k)} e^{\frac{2 n \pi z}{k^2}} P(x)^{\alpha}\,dz, \end{align*} from which the modified functional equation from Theorem \ref{functional equation} yields \begin{align*} p_\alpha(n) = \sum_{h, k} i k^{-\frac{\alpha}{2} - 2} e^{-\frac{2 \pi i n h}{k}} \omega^{(\alpha)}(h, k) \int_{z_1(h,k)}^{z_2(h,k)} e^{\frac{2 \pi n z}{k^2}} \Psi_k^{(\alpha)}(z) P(x')^{\alpha} \,dz, \end{align*} where \[\omega^{(\alpha)}(h, k) := e^{\alpha \pi i s(h, k)}, \quad \text{and} \quad \Psi_k^{(\alpha)}(z) := z^{\frac{\alpha}{2}} \exp\Big(\frac{\alpha \pi}{12 z} - \frac{\alpha \pi z}{12 k^2}\Big). \]

Let \(q = \lfloor \alpha/24 \rfloor\), and define \begin{equation*}
    Q^{(\alpha)}(x) := \sum_{m=0}^q p_\alpha(m) x^m.
\end{equation*} We proceed by separating out a part of the integral that corresponds to \(Q^{(\alpha)}(x)\) and showing that the remaining part goes to zero as \(N \to \infty\). In particular, we write \begin{align*}
    I_1(h,k) = \int_{z_1(h, k)}^{z_2(h, k)} \Psi_k^{(\alpha)}(z) e^{\frac{2 \pi n z}{k^2}} Q^{(\alpha)}(x') dz
\end{align*} and \begin{align*} \quad I_2(h,k) = \int_{z_1(h, k)}^{z_2(h, k)} \Psi_k^{(\alpha)}(z) e^{\frac{2 \pi n z}{k^2}}(P(x')^{\alpha} - Q^{(\alpha)}(x')) \,dz
\end{align*} to obtain \begin{align}\label{palphasum}
    p_{\alpha}(n) = \sum_{h, k}i k^{-2 - \frac{\alpha}{2}} e^{-\frac{2 \pi i n h}{k}} \omega^{(\alpha)}(h, k) \cdot (I_1(h, k) + I_2(h, k)).
\end{align}

We now show that \(I_2(h, k)\) is ``small'' for large \(N\) by considering the integral along the chord in the \(z\)-plane joining \(z_1(h, k)\) and \(z_2(h, k)\). Because $0 < \Re z \leq 1$ and $\Re(z^{-1}) \geq 1$ for \(z\) on the path of integration, we can write \begin{align}
   & \left| \Psi_{k}^{(\alpha)} (z) \cdot e^{\frac{2 n \pi z}{k^2}} \cdot \left\{ P(x')^{\alpha} - Q^{(\alpha)}(x')\right\} \right| \\
    & = |z|^{\frac{\alpha}{2}} \exp \left( \frac{\alpha \pi}{12} \Re(z^{-1}) - \frac{\alpha\pi}{12 k^2} \Re z + \frac{2 n \pi}{k^2}\Re z \right) \cdot \left| \sum_{m=q+1}^\infty p_\alpha(m) \exp(\frac{2 \pi i H m}{k} - \frac{2 \pi m}{z}) \right| \nonumber \\
    & \leq |z|^{\frac{\alpha}{2}} \exp \left( \frac{\alpha \pi}{12} \Re(z^{-1}) + \frac{2 n \pi}{k^2}\right) \sum_{m=q+1}^\infty p_\alpha(m) e^{-2 \pi m \Re(z^{-1})} \nonumber \\
    & \leq |z|^{\frac{\alpha}{2}} \sum_{m=q+1}^{\infty}p_\alpha(m)e^{-2\pi\left(m - \frac{\alpha}{24}\right)\Re(z^{-1})} \\
    &\leq |z|^{\frac{\alpha}{2}} \sum_{m=q+1}^{\infty}p_\alpha(m)e^{-2\pi\left(m - \frac{\alpha}{24}\right)} = |z|^{\frac{\alpha}{2}} e^{\frac{\alpha\pi}{12}} (P(e^{-2\pi})^\alpha - Q^{(\alpha)}(e^{-2\pi})).
\end{align} Since $|z| < \sqrt{2} k / N$ for \(z\) on the chord from \(z_1(h, k)\) to \(z_2(h, k)\), the integrand is less than $C (k / N)^{\alpha / 2}$ for some constant $C$ not depending on $N$. Thus, because the length of the chord is at most $2 \sqrt{2} k/N$, we have \begin{align}\label{I_2}
    |I_2(h, k)| & < \frac{C k^{\frac{\alpha}{2}+1}}{N^{\frac{\alpha}{2}+1}}.
\end{align} Substituting this bound into the sum of the \(I_2\) terms in (\ref{palphasum}) yields \begin{align*}
    \left| \sum_{h, k} i k^{- \frac{\alpha}{2} - 2} e^{-\frac{2 \pi i n h}{k}} \omega^{(\alpha)} (h, k)  I_2(h, k) \right| & < \sum_{k = 1}^N \sum_{\substack{0 \leq h < k \\ (h, k) = 1}} C k^{-1} N^{-\frac{\alpha}{2}-1} \leq C N^{-\frac{\alpha}{2}-1} \sum_{k=1}^N 1 = CN^{-\frac{\alpha}{2}}.
\end{align*} Thus, we have \begin{align}\label{simpi1i2}
    p_\alpha(n) & = \bigg(\sum_{k = 1}^N \sum_{\substack{0 \leq h < k \\ (h, k) = 1}} i k^{-\frac{\alpha}{2}-2} e^{-\frac{2 \pi i n h}{k}}\omega^{(a)}(h, k) I_1(h, k)\bigg) + O(N^{-\frac{\alpha}{2}}).
\end{align}

Next we consider $I_1(h, k)$. We can write \begin{align}\label{i1}
    I_1(h,k) = \int_{-K} - \int_0^{z_1(h, k)} - \int_{z_2(h, k)}^0 =: \mathord{\int_{-K}} - J_1 - J_2,
\end{align} where we omit the integrands for brevity, and where $-K$ indicates that we integrate in the negative direction along $K$. Because \(|z| \leq \sqrt{2}k/N\) on the paths of integration, we can bound the integrands of \(J_1\) and \(J_2\) by \begin{align} \label{J}
    &\left| \Psi^{(\alpha)}_k(z) e^{\frac{2 \pi n z}{k^2}} Q^{(\alpha)}(x') \right| \nonumber \\
    &\leq |z|^{\frac{\alpha}{2}} \exp\left( \frac{\alpha\pi}{12} \Re(z^{-1}) - \frac{\alpha\pi}{12 k^2} \Re z + \frac{2n \pi}{k^2}\Re z \right) \left| \sum_{m=0}^q p_\alpha(m) \exp(\frac{2 \pi i H m}{k} - \frac{2 \pi m}{z}) \right| \nonumber \\
    &\leq |z|^{\frac{\alpha}{2}} \exp\left( \frac{\alpha\pi}{12} + \frac{2\pi}{k^2} \left( n - \frac{\alpha}{24} \right) \Re z \right) \left| \sum_{m=0}^q p_\alpha(m) e^{-2\pi m} \right| \\
    & \leq \frac{e^{2n\pi} 2^{\frac{\alpha}{4}} k^{\frac{\alpha}{2}}}{N^{\frac{\alpha}{2}}}\left| \sum_{m=0}^q p_\alpha(m) e^{-2\pi m} \right|.
\end{align} The lengths of the arcs from 0 to $z_1(h,k)$ and $z_2(h,k)$ are less than $\pi|z_1(h, k) |$ and $\pi|z_2(h, k)|$, respectively, and both of these are bounded by \(\pi \sqrt{2} k/N\), so we get that $|J_1|, |J_2| < C_1 k^{\frac{\alpha}{2}+1} N^{-\frac{\alpha}{2}-1}$ for some constant \(C_1\).
 
Combining (\ref{simpi1i2}), (\ref{i1}), and the bounds for \(J_1\) and \(J_2\) above, we find that \begin{align}\label{finite sum}
    p_\alpha(n) & = \sum_{k = 1}^N \sum_{\substack{0 \leq h < k \\ (h, k) = 1}} i k ^{-\frac{\alpha}{2}-2} e^{-\frac{2 \pi i nh}{k}}\omega^{(\alpha)}(h, k) \int_{-K} \Psi_k^{(\alpha)} (z) e^{\frac{2n \pi z}{k^2}} Q^{(\alpha)}(x')\,dz + O(N^{-\frac{\alpha}{2}}),
\end{align} which in the limit as $N$ goes to infinity becomes \begin{align*}
    p_\alpha(n) & = \sum_{m=0}^q p_\alpha(m) \sum_{k = 1}^\infty \sum_{\substack{0 \leq h < k \\ (h, k) = 1}} i k ^{-\frac{\alpha}{2}-2} e^{-\frac{2 \pi i nh}{k}}\omega^{(\alpha)}(h, k) \\
    & \quad \cdot \int_{-K} z^{\frac{\alpha}{2}}\exp\left(\frac{2\pi nz}{k^2} + \frac{\alpha\pi}{12z} - \frac{\alpha\pi z}{12k^2} + \frac{2\pi imH}{k} - \frac{2\pi m}{z} \right) \,dz \\
    &= \sum_{m=0}^q p_\alpha(m) \sum_{k = 1}^\infty \sum_{\substack{0 \leq h < k \\ (h, k) = 1}} i k ^{-\frac{\alpha}{2}-2} e^{\frac{2\pi i}{k}(mH-nh)}\omega^{(\alpha)}(h, k) \\
    & \quad \cdot \int_{-K} z^{\frac{\alpha}{2}}\exp\left(\frac{2\pi z}{k^2} \nu_\alpha(n)^2 + \frac{2\pi}{z} \mu_\alpha(m)^2\right) \,dz \\
    &= \sum_{m=0}^q p_\alpha(m) \sum_{k = 1}^\infty i\frac{A_k^{(\alpha)}(n, m)}{k^{\frac{\alpha}{2}+2}} \int_{-K} z^{\frac{\alpha}{2}}\exp\left(\frac{2\pi z}{k^2} \nu_\alpha(n)^2 + \frac{2\pi}{z} \mu_\alpha(m)^2\right) \,dz.
\end{align*} To evaluate the integral on the right, we make the change of variables $t = 2\pi(\alpha/24 - m)/z$ to obtain \begin{align*}
    p_\alpha(n) & = 2\pi \sum_{m=0}^q p_\alpha(m) \sum_{k = 1}^\infty \frac{A_k^{(\alpha)}(n, m)}{k^{\frac{\alpha}{2}+2}} \left[ 2\pi\mu_\alpha(n)^2 \right]^{\frac{\alpha}{2}+1} \\
    & \quad \cdot \frac{1}{2\pi i} \int_{c-i\infty}^{c+i\infty} t^{-\frac{\alpha}{2}-2}\exp\left(t + \left(\frac{2\pi}{k} \nu_\alpha(n) \mu_\alpha(m)\right)^2 \frac{1}{t} \right) \,dz,
\end{align*} where $c = \alpha \pi/12$. Now recall that the modified Bessel function of the first kind satisfies \begin{align}
    I_\beta(z) & = \frac{(z/2)^\beta}{2 \pi i} \int_{c - \infty i}^{c + \infty i} t^{-\beta - 1} e^{t + \frac{z^2}{4t}} \, dt 
\end{align} for $c > 0, \Re(\nu) > 0$ \cite[p. 181]{watson}. Consequently, for $n \geq \alpha/24$, we find that \begin{align}
    p_\alpha(n) &= 2\pi \sum_{m=0}^q p_\alpha(m) \sum_{k = 1}^\infty \frac{A_k^{(\alpha)}(n, m)}{k^{\frac{\alpha}{2}+2}} \left( 2\pi\mu_\alpha(m)^2 \right)^{\frac{\alpha}{2}+1} \left(\frac{2\pi}{k} \nu_\alpha(n)\mu_\alpha(m)\right)^{-\frac{\alpha}{2}-1} I_{\frac{\alpha}{2}+1}\left(\frac{4\pi}{k}\nu_\alpha(n)\mu_\alpha(m)\right) \\
    &= \nu_\alpha(n)^{-\frac{\alpha}{2}-1}\sum_{m=0}^q \mu_\alpha(m)^{\frac{\alpha}{2}+1} p_\alpha(m) \sum_{k = 1}^\infty \frac{2\pi}{k} A_k^{(\alpha)}(n, m)I_{\frac{\alpha}{2}+1}\left(\frac{4\pi}{k}\nu_\alpha(n)\mu_\alpha(m)\right).
\end{align}

\end{proof}

\section{Applications of the Series Formula for \(p_\alpha(n)\)}

\subsection{Estimates of $p_\alpha(n)$}

In this section, we consider the error of the approximation \begin{equation}
    p_\alpha(n; \delta) := \nu_\alpha(n)^{-\frac{\alpha}{2}-1}\sum_{m=0}^q \mu_\alpha(m)^{\frac{\alpha}{2}+1} p_\alpha(m) \sum_{1 \leq k < \frac{2\pi}{\delta}\mu_\alpha(m)} \frac{2\pi}{k} A_k^{(\alpha)}(n,m)I_{\frac{\alpha}{2}+1}\left(\frac{4\pi}{k}\nu_\alpha(n)\mu_\alpha(m)\right).
\end{equation} for \(p_\alpha(n)\). Note in particular that in the limit as \(\delta \to 0^+\), we have \(p_\alpha(n; \delta) \to p_\alpha(n)\).

\begin{thm} \label{seriesboundlem}
For all \(\alpha > 0\), \(0 < \delta < 2\pi\mu_\alpha(0)\), and \(n > \alpha/24\), we have \begin{align}
    \left|p_\alpha(n) - p_\alpha(n; \delta)\right| &< \frac{C}{\delta} \frac{I_{\frac{\alpha}{2}+1}(2\delta\nu_\alpha(n))}{\nu_\alpha(n)^{\frac{\alpha}{2}+1}} < C\delta^{\frac{\alpha}{2}} \frac{I_{\frac{\alpha}{2}+1}(4\pi\mu_\alpha(0)\nu_\alpha(n))}{(2\pi\mu_\alpha(0)\nu_\alpha(n))^{\frac{\alpha}{2}+1}},
\end{align} where \begin{equation*}
    C := 4\pi^2\left( 1 + \frac{2}{\alpha} \right)\mu_\alpha(0) \sum_{m=0}^q \mu_\alpha(m)^{\frac{\alpha}{2}+1}p_\alpha(m).
\end{equation*}

\end{thm}

\begin{proof}
Start by noting that \begin{equation}
    \left| A_k^{(\alpha)}(n, m) \right| \leq \sum_{\substack{0 \leq h < k \\ (h, k) = 1}} \left| e^{\alpha \pi i s(h, k) + \frac{2 \pi i}{k}(mH - nh)} \right| = \sum_{\substack{0 \leq h < k \\ (h, k) = 1}} 1 \leq k.
\end{equation} Moreover, using the fact from \cite{paris} that for $0 < x < y$ and $\nu > 1$, the modified Bessel function of the first kind satisfies
\begin{align}\label{paris}
    \frac{I_\nu(x)}{I_\nu(y)} < \left( \frac{x}{y} \right)^{\nu},
\end{align} we have that \begin{align*}
    \sum_{k \geq \frac{2\pi}{\delta}\mu_\alpha(m)} \frac{I_{\frac{\alpha}{2}+1}\left(\frac{4\pi}{k}\nu_\alpha(n)\mu_\alpha(m)\right)}{I_{\frac{\alpha}{2}+1}(2\delta\nu_\alpha(n))} & < \sum_{k \geq \frac{2\pi}{\delta}\mu_\alpha(m)} \left(\frac{2\pi}{k\delta}\mu_\alpha(m)\right)^{\frac{\alpha}{2}+1} \\ & < 1 + \int_{\frac{2\pi}{\delta}\mu_\alpha(m)}^\infty \left(\frac{2\pi}{t\delta}\mu_\alpha(m)\right)^{\frac{\alpha}{2}+1}\,dt \\ & = 1 + \frac{4\pi}{\alpha\delta}\mu_\alpha(m)
\end{align*} for \(0 \leq m \leq q\). Thus, we find that \begin{align*}
    \nu_\alpha(n)^{\frac{\alpha}{2}+1}\left|p_\alpha(n) - p_\alpha(n; \delta)\right| &\leq 2\pi \sum_{m=0}^q \mu_\alpha(m)^{\frac{\alpha}{2}+1}p_\alpha(m) \sum_{k \geq \frac{2\pi}{\delta}\mu_\alpha(m)} I_{\frac{\alpha}{2}+1}\left(\frac{4\pi}{k}\nu_\alpha(n)\mu_\alpha(m)\right) \\
    &< 2\pi I_{\frac{\alpha}{2}+1}(2\delta\nu_\alpha(n)) \sum_{m=0}^q \mu_\alpha(m)^{\frac{\alpha}{2}+1}p_\alpha(m) \left[ 1 + \frac{4\pi}{\alpha\delta}\mu_\alpha(m) \right].
\end{align*} Since \(1 < \frac{2\pi}{\delta}\mu_\alpha(0)\) and \(\mu_\alpha(m) \leq \mu_\alpha(0)\), it follows that \begin{align*}
    |p_\alpha(n) - p_\alpha(n; \delta)| &< \frac{4\pi^2}{\delta} \frac{I_{\frac{\alpha}{2}+1}(2\delta\nu_\alpha(n))}{\nu_\alpha(n)^{\frac{\alpha}{2}+1}} \left( 1 + \frac{2}{\alpha} \right)\mu_\alpha(0) \sum_{m=0}^q \mu_\alpha(m)^{\frac{\alpha}{2}+1}p_\alpha(m),
\end{align*} or applying the Paris inequality a second time using \(2\delta\nu_\alpha(n) < 4\pi\mu_\alpha(0)\nu_\alpha(n)\), \begin{align*}
    |p_\alpha(n) - p_\alpha(n; \delta)| &< 4\pi^2 \delta^{\frac{\alpha}{2}} \frac{I_{\frac{\alpha}{2}+1}(4\pi\mu_\alpha(0)\nu_\alpha(n))}{(2\pi\mu_\alpha(0)\nu_\alpha(n))^{\frac{\alpha}{2}+1}}  \left( 1 + \frac{2}{\alpha} \right)\mu_\alpha(0) \sum_{m=0}^q \mu_\alpha(m)^{\frac{\alpha}{2}+1}p_\alpha(m).
\end{align*}
\end{proof}

\noindent
We are now in a position to prove the simple asymptotic formula for \(p_\alpha(n)\) stated in the introduction.

\begin{proof}[Proof of Corollary~\ref{asymp}]

Observe that since \(\frac{4\pi}{\delta}\mu_\alpha(m)\) is strictly increasing in \(m\), there exists a \(0 < \delta < 2\pi\mu_\alpha(0)\) such that \(\frac{2\pi}{\delta}\mu_\alpha(m) \leq 2\) for \(0 < m \leq q\) and so \begin{equation*}
    p_\alpha(n; \delta) = 2\pi\left(\frac{\mu_\alpha(0)}{\nu_\alpha(n)}\right)^{\frac{\alpha}{2}+1} I_{\frac{\alpha}{2}+1}\left(4\pi\nu_\alpha(n)\mu_\alpha(0)\right) = 2\pi \frac{I_{\frac{\alpha}{2}+1}\left(\frac{\pi\alpha}{6} \lambda_\alpha(n) \right)}{\lambda_\alpha(n)^{\frac{\alpha}{2}+1}}.
\end{equation*} Moreover, by Theorem \ref{seriesboundlem}, we have \begin{equation*}
    |p_\alpha(n) - p_\alpha(n; \delta)| \leq C \frac{I_{\frac{\alpha}{2}+1}(2\delta\nu_\alpha(n))}{\nu_\alpha(n)^{\frac{\alpha}{2}+1}}
\end{equation*} for some constant \(C\). Using the fact that \(I_\nu(z) \sim e^z/\sqrt{2 \pi z}\) from \cite[10.30.4]{nist}, we easily verify that \(C\nu_\alpha(n)^{-\frac{\alpha}{2}-1}I_{\frac{\alpha}{2}+1}(2\delta\nu_\alpha(n)) \ll p_\alpha(n)\), from which it follows that \begin{equation*}
    p_\alpha(n) \sim p_\alpha(n; \delta) \sim \frac{ e^{\frac{\alpha\pi}{6}\lambda_\alpha(n)}}{\lambda_\alpha(n)^{\frac{\alpha + 3}{2}}}.
\end{equation*}

\end{proof}

\noindent
Theorem \ref{seriesboundlem} also allows us to derive a finite exact formula for \(p_\alpha(n)\) when \(\alpha\) is rational. This is made possible by a formula for the denominator of \(p_\alpha(n)\) from \cite{chan}, which states that if \(\alpha = a/b\) for coprime \(a, b \in \Z\) with $b > 0$, then \begin{equation*}
    \operatorname{denom}(p_\alpha(n)) := b^n \prod_{p \mid b} p^{\operatorname{ord}_p(n!)},
\end{equation*} where \(\operatorname{ord}_p(n)\) denotes the multiplicity of a prime \(p\) as a factor of \(n\).

\begin{cor}\label{rational}
Let $\alpha, \varepsilon > 0$ and $n > \alpha/24$ with $\alpha$ rational. Then \begin{equation}
    p_\alpha(n) = \frac{\lfloor Dp_\alpha(n; \delta) \rceil}{D},
\end{equation} where \(D = \operatorname{denom}(p_\alpha(n))\) and \begin{equation*}
    \delta := \left( \frac{(2\pi\mu_\alpha(0)\nu_\alpha(n))^{\frac{\alpha}{2}+1}}{2DCI_{\frac{\alpha}{2}+1}(4\pi\mu_\alpha(0)\nu_\alpha(n))} \right)^{\frac{2}{\alpha}},
\end{equation*} with \(C\) defined as in Theorem \ref{seriesboundlem}.
\end{cor}

\begin{proof}
Observe that by Theorem~\ref{seriesboundlem}, we have \begin{align}
    |p_\alpha(n) - p_\alpha(n; \delta)| &< C\delta^{\frac{\alpha}{2}} \frac{I_{\frac{\alpha}{2}+1}(4\pi\mu_\alpha(0)\nu_\alpha(n))}{(2\pi\mu_\alpha(0)\nu_\alpha(n))^{\frac{\alpha}{2}+1}} = \frac{1}{2D}.
\end{align} Thus, \(D|p_\alpha(n) - p_\alpha(n; \delta)| < 1/2\), implying that \(Dp_\alpha(n)\) is the nearest integer to \(Dp_\alpha(n; \delta)\).
\end{proof}

\subsection{Hyperbolicity of the Jensen Polynomials of $p_\alpha(n)$}

In this section, we demonstrate how the asymptotics of $p_{\alpha}(n)$ in this paper can be used to generalize a recent hyperbolicity result for the usual partition function. 

\begin{proof}[Proof of Theorem~\ref{jensen}]
Set \begin{equation}
    m = \frac{\alpha}{24} \quad \text{and} \quad c_0 = \log \left(\sqrt{\frac{12}{\alpha}} \cdot \left(\frac{\alpha}{24}\right)^{\frac{\alpha+3}{4}}\right).\nonumber
\end{equation} Then by Corollary~\ref{asymp},  \begin{align*}
    p_\alpha(n) \sim e^{c_0 + 4 \pi \sqrt{mn}} n^{-\frac{\alpha + 3}{4}}.
\end{align*} Thus, as in \cite[\S3]{griffin}, we have \begin{align*}
    \log \left( \frac{p_\alpha(n + j)}{p_\alpha(n)} \right) \sim 4 \pi \sqrt{m} \sum_{i = 1}^\infty \binom{1/2}{i} \frac{j^i}{n^{i - 1/2}} - \frac{\alpha +3}{4} \sum_{i = 1}^\infty \frac{(-1)^{i - 1}j^i}{i n^i},
\end{align*} from which it is clear that $p_\alpha(n)$ satisfies the conditions of Theorem 3 from \cite{griffin} with $A(n) = 2\pi \sqrt{m/n} + O(1 / n)$ and $\delta(n) = (\pi / 2)^{1/2} m^{1/4} n^{-3/4} + O(n^{-5/4})$. It follows immediately that for all \(d\) the Jensen polynomials associated with \(p_\alpha(n)\) are hyperbolic for sufficiently large \(n\).
\end{proof}

\begin{remark}
The proof of Theorem \ref{jensen} follows \cite[\S3]{griffin}. In particular, we consider the renormalization of the Jensen polynomials given by \begin{align}
\widehat{J}_{p_{\alpha}}^{d, n}(X) & = \frac{\delta(n)^{-d}}{p_\alpha(n)} \cdot J_{p_\alpha}^{d, n}\left( \frac{\delta(n)X - 1}{\exp(A(n))} \right).
\end{align} Theorem \ref{jensen} follows from the fact that for fixed $d$, \begin{align}
    \lim_{n \to \infty} \widehat{J}_{p_{\alpha}}^{d, n}(X) = H_d(x),
\end{align} where $H_d(x)$ is the degree $d$ renormalized Hermite polynomial in \cite{griffin}.
\end{remark}

\section[Numerical Data]{Numerical Data\footnote{All computations in this section were done with Wolfram Mathematica.}}

In this section, we illustrate the theorems of the previous sections using numerical examples. For simplicity, we limit our examples to cases where \(0 < \alpha < 24\). For such \(\alpha\), it will be convenient to define \begin{align}
    r_\alpha(n; m) & = \frac{\Re\left(p_\alpha\left(n; \frac{2\pi\mu_\alpha(0)}{m+1}\right)\right)}{p_\alpha(n)},
\end{align} the ratio between the real part of the \(m\)-term approximation to \(p_\alpha(n)\) and the actual value. Note that a value of $r_\alpha(n; m)$ closer to 1 indicates that the $m$-term approximation to $r_\alpha(n)$ is more accurate.

By Corollary \ref{asymp}, we know that $p_\alpha(n)$ is asymptotically equivalent to the first term in the series expansion in Theorem~\ref{main theorem} as $n$ goes to infinity. Table~\ref{e, n} displays the accuracy of the first-term expansion for $\alpha = e$ and $n$ varying from 1 to 10. 
Table~\ref{table1} shows the ratio of both the first-term and the five-term approximation to $p_{\alpha}(n)$ where $\alpha = 1/\pi$ and $\alpha = 5$. Note that the sign of the error term \(|p_\alpha(n)-p_\alpha(n; m)|\) is usually periodic with period \(m+1\). This is a consequence of the periodicity of the Kloosterman sums. \\
\FloatBarrier
\begin{table}
\parbox{.45\linewidth}{
\centering
\begin{tabular}{|l|l|l|l|}
\hline $n$ & $p_e(n)$ & $\Re(p_e(n; 1))$ & $r_e(n; 1)$ \\
\hline 
1  & 2.71      & 2.83      & 1.04253 \\
2  & 7.77      & 7.65      & 0.98444 \\
3  & 18.05     & 18.23     & 1.01014 \\
4  & 40.26     & 39.96     & 0.99263 \\
5  & 81.84     & 82.28     & 1.00543 \\
6  & 161.99    & 161.41    & 0.9964  \\
7  & 303.75    & 304.41    & 1.00217 \\
8  & 556.32    & 555.61    & 0.99873 \\
9  & 985.41    & 986.27    & 1.00086 \\
10 & 1710.31   & 1709.07   & 0.99927 \\ \hline
\end{tabular}
\caption{Accuracy of first-term approximations to $p_e(n)$}\label{e, n}
}
\hfill
\parbox{.45\linewidth}{
\centering
\begin{tabular}{|l|l|l|}
\hline 
$m$ & $\Re(p_{1/e}(50; m))$ & $r_{1/e}(50; m)$\\ 
\hline
1  & 356.2898 & 0.997668 \\
2  & 357.2586 & 1.000381 \\
3  & 357.1278 & 1.000014 \\
4  & 357.053  & 0.999805 \\
5  & 357.1236 & 1.000003 \\
6  & 357.1195 & 0.999991 \\
7  & 357.1169 & 0.999984 \\
8  & 357.1208 & 0.999995 \\
9  & 357.1201 & 0.999993 \\
10 & 357.1296 & 1.00002 \\ \hline
\end{tabular}
\caption{Accuracy of $m$-term approximations to $p_{1/e}(50) = 357.1225$}\label{1/e, m}
}
\end{table}

\begin{table}[htb!]
\begin{center}
\begin{tabular}{|l|l|l|l|l|}
\hline
$n$ & $r_{1/\pi}(n; 1)$ & $r_{1/\pi}(n; 5)$ & $r_{5}(n; 1)$ & $r_{5}(n; 5)$ \\
\hline
1  & 1.294180591 & 0.953980957 & 1.015286846 & 1.000097277 \\
2  & 0.970982400 & 0.982523054 & 0.994583967 & 1.000042848 \\
3  & 1.083673986 & 1.018088216 & 1.002732222 & 1.000007177 \\
4  & 0.923295102 & 1.02170408  & 0.998466124 & 0.999992664 \\
5  & 1.124698668 & 1.016001474 & 1.000871244 & 0.999999382 \\
6  & 0.897139773 & 1.004350338 & 0.999524823 & 1.000000088 \\
7  & 1.108496000 & 0.978153497 & 1.000255655 & 1.000000217 \\
8  & 0.943494666 & 1.002688299 & 0.999854031 & 1.000000092 \\
9  & 1.034408356 & 1.003218418 & 1.000093623 & 0.999999982 \\
10 & 0.961090657 & 1.005487344 & 0.999935881 & 0.999999997 \\
11 & 1.076769973 & 0.993996646 & 1.000043109 & 0.999999968 \\
12 & 0.923558631 & 1.005396386 & 0.999972215 & 1.000000007 \\
13 & 1.058750442 & 0.996292489 & 1.000017874 & 1.000000008 \\
14 & 0.980265489 & 0.993723758 & 0.999987986 & 1.000000000 \\
 \hline
\end{tabular}
\caption{Accuracy of approximation to $p_\alpha(n)$ as $n$ increases}\label{table1}
\end{center}
\end{table}

Table~\ref{1/e, m} displays how $p_\alpha(n,m)$ converges to $p_\alpha(n)$ for $\alpha = 1/e, n = 50$, and $1 \leq m \leq 10$. Table~\ref{table2} displays the ratio of the $m$-term approximation of $p_\alpha(n)$ to the actual value for $n = 100$ and various values of $\alpha$ and $m$. As we increase $\alpha$, we see that the relative error of the approximation for $p_\alpha(n)$ decreases. 

\begin{table}[htb!]
\begin{center}
\begin{tabular}{|l|l|l|l|l|}
\hline
$m$ & $r_{0.01}(100; m)$ & $r_{0.1}(100; m)$  & $r_{1}(100; m)$ & $r_{10}(100; m)$  \\
\hline
1  & 0.846079580 & 0.988058877 & 0.999998178 & 1.000000000 \\
2  & 0.969774117 & 0.999386989 & 1.000000009 & 1.000000000  \\
3  & 0.920711483 & 0.997246602 & 0.999999995 & 1.000000000  \\
4  & 0.973881495 & 0.999016179 & 0.999999999 & 1.000000000  \\
5  & 1.040636574 & 1.000923931 & 1.000000000 & 1.000000000  \\
6  & 1.028999226 & 1.000579623 & 1.000000000 & 1.000000000  \\
7  & 1.020829553 & 1.000421683 & 1.000000000 & 1.000000000  \\
8  & 0.995326778 & 0.999817677 & 1.000000000 & 1.000000000  \\
9  & 0.995461037 & 0.999846688 & 1.000000000 & 1.000000000  \\
10 & 1.011689149 & 1.000211135 & 1.000000000 & 1.000000000  \\
\hline
\end{tabular}
\caption{Accuracy of approximation to $p_\alpha(n)$ as number of terms in series increases}\label{table2}
\end{center}
\end{table}

Table~\ref{hermitetable2} depicts the convergence of $\widehat{J}_{p_{\alpha}}^{2, n}(X)$ to the Hermite polynomial $H_2(x) = x^2 - 2$, and the convergence of $\widehat{J}_{p_{\alpha}}^{3, n}(X)$ to the Hermite polynomial $H_3(x) = x^3 - 6x$. Here, \begin{align*}
    A(n) = 2\pi \sqrt{\frac{\alpha}{24n-\alpha}} - \frac{24}{24n-\alpha}, \quad \text{and} \quad \delta(n) = \sqrt{\frac{12\pi\alpha^{\frac{1}{2}}}{(24n-\alpha)^{\frac{3}{2}}} - \frac{288\alpha}{(24n - \alpha)^2}}
\end{align*} as in Theorem~\ref{jensen}, for $\sqrt{3}$. To compute $p_\alpha(n)$ for large $n$, we used the 100-term approximation of our series formula; this is valid for our purposes because by Theorem~\ref{seriesboundlem}, the relative error \(|r_{\sqrt{3}}(n, 100) - 1|\) is bounded by \(10^{-75}\) for the values of \(n\) we consider.

\begin{table}[htb!]
\begin{center}
\begin{tabular}{|l|l|l|}
\hline & & \\[-1em]
    $n$ & $\widehat{J}_{p_{\sqrt{3}}}^{2, n}(x)$ & $\widehat{J}_{p_{\sqrt{3}}}^{3, n} (x)$ \\
    \hline
10000 & $0.999598 x^2 + 0.120905 x -2.03828$    & $ 0.999942 x^3 + 0.0939817 x^2 - 6.03526 x - 0.648632$ \\
20000 & $0.999804 x^2 + 0.0966267 x -2.02711   $ & $0.999971 x^3 + 0.0767061 x^2 - 6.02522 x - 0.543473$  \\
30000 & $0.999871 x^2 + 0.0852795 x - 2.02216$    & $0.999981 x^3 + 0.0683801 x^2 - 6.0207 x - 0.495049$ \\
40000 & $0.999904 x^2 + 0.0782302 x -2.0192$ & $0.999986 x^3 + 0.0631174 x^2 - 6.01799 x - 0.435239$  \\
50000 & $ 0.999923 x^2 + 0.0732538 x - 2.01719$  & $1.00252 x^3 + 0.0595086 x^2 - 6.03131 x - 0.429626$ \\ 
\quad $\vdots$ & \multicolumn{1}{c|}{\vdots} & \multicolumn{1}{c|}{\vdots} \\
\hspace{0.11 cm} $\infty$ & \multicolumn{1}{c|}{$x^2 - 2$} & \multicolumn{1}{c|}{$x^3 - 6x$} \\
\hline
\end{tabular}
\caption{Convergence to the Hermite polynomial of degree 2, $x^2 - 2$, and of degree 3, $x^3 - 6x$}\label{hermitetable2}
\end{center}
\end{table}

\FloatBarrier
In Table~\ref{rationalalpha}, we provide the actual value of $p_{51/7}(n)$ alongside the minimum number \(M_{51/7}(n)\) for which Corollary~\ref{rational} guarantees that \(p_{51/7}(n)\) is given by a suitable rounding of \(p_\alpha\left(n; \frac{2\pi\mu_\alpha(0)}{M_{51/7}(n)+1}\right)\), which has $M_{51/7}(n)$ terms. We also provide $M^\ast_{51/7}(n)$, the minimum number of terms such that this is numerically true.

\begin{table}[htb!]
\begin{center}
\begin{tabular}{|l|l|l|l|}
\hline
    & \vspace{-10 pt} & & \\
    $n$ & $p_{51/7}(n)$ & $M_{51/7}(n;D)$ & $M^*_{51/7}(n;D)$ \\
\hline
1  & 51/7                       & 2   & 1   \\
2  & 1836/49                    & 3   & 2   \\
3  & 52751/343                  & 5   & 3   \\
4  & 1322226/2401               & 8   & 4   \\
5  & 29852442/16807             & 14  & 7   \\
6  & 623075585/117649           & 23  & 10  \\
7  & 85346705106/5764801        & 67  & 26  \\
8  & 1583888229297/40353607     & 114 & 43  \\
9  & 28093059550223/282475249   & 194 & 63  \\
10 & 479246612549889/1977326743 & 330 & 109 \\
\hline
\end{tabular}
\caption{Number of terms for exact solution for $p_{51/7}(n)$}\label{rationalalpha}
\end{center}
\end{table}

\Addresses

\clearpage
\clearpage

\begin{thebibliography}{10}

\bibitem{apostol}
Tom~M. Apostol.
\newblock {\em Modular {Functions} and {Dirichlet} {Series} in {Number}
  {Theory}}.
\newblock Graduate {Texts} in {Mathematics}. Springer-Verlag, New York, 2nd
  edition, 1990.

\bibitem{bevilacqua}
Erin Bevilacqua, Kapil Chandran, and Yunseo Choi.
\newblock Ramanujan {Congruences} for {Fractional} {Partition} {Functions}.
\newblock Unpublished, 2019.

\bibitem{bringmann}
Kathrin Bringmann, Amanda Folsom, Ken Ono, and Larry Rolen.
\newblock {\em Harmonic {M}aass forms and mock modular forms: theory and
  applications}, volume~64 of {\em American Mathematical Society Colloquium
  Publications}.
\newblock American Mathematical Society, Providence, RI, 2017.

\bibitem{chan}
Heng~Huat Chan and Liuquan Wang.
\newblock Fractional powers of the generating function for the partition
  function.
\newblock {\em Acta Arithmetica}, 187:59--80, 2019.

\bibitem{chen}
William Chen, Dennis Jia, and Larry Wang.
\newblock Higher order {Tur\'an} inequalities for the partition function.
\newblock {\em Transactions of the American Mathematical Society}, 2018.

\bibitem{craven1989jensen}
Thomas Craven and George Csordas.
\newblock Jensen polynomials and the {T}ur{\'a}n and {Laguerre} inequalities.
\newblock {\em Pacific Journal of Mathematics}, 136(2):241--260, 1989.

\bibitem{DeSalvo}
Stephen DeSalvo and Igor Pak.
\newblock Log-concavity of the partition function.
\newblock {\em The Ramanujan Journal}, 38(1):61--73, Oct 2015.

\bibitem{nist}
{\it NIST Digital Library of Mathematical Functions}.
\newblock http://dlmf.nist.gov/, Release 1.0.23 of 2019-06-15.
\newblock F.~W.~J. Olver, A.~B. {Olde Daalhuis}, D.~W. Lozier, B.~I. Schneider,
  R.~F. Boisvert, C.~W. Clark, B.~R. Miller and B.~V. Saunders, eds.

\bibitem{griffin}
Michael Griffin, Ken Ono, Larry Rolen, and Don Zagier.
\newblock Jensen polynomials for the {Riemann} zeta function and other
  sequences.
\newblock {\em Proceedings of the National Academy of Sciences},
  116(23):11103--11110, 2019.

\bibitem{hardy}
Godfrey~H. Hardy and Srinivasa Ramanujan.
\newblock Asymptotic formul{\ae} in combinatory analysis.
\newblock {\em Proceedings of the London Mathematical Society}, 2(1):75--115,
  1918.

\bibitem{helfgott}
H.~A. {Helfgott}.
\newblock {The ternary {Goldbach} conjecture is true}.
\newblock {\em arXiv e-prints}, page arXiv:1312.7748, Dec 2013.

\bibitem{keith}
William~J. Keith.
\newblock Restricted {$k$}-color partitions.
\newblock {\em Ramanujan J.}, 40(1):71--92, 2016.

\bibitem{larson}
Hannah Larson and Ian Wagner.
\newblock Hyperbolicity of the partition {Jensen} polynomials.
\newblock {\em Research in Number Theory}, 5(2):19, June 2019.

\bibitem{nicolas}
Jean-Louis Nicolas.
\newblock Sur les entiers $n$ pour lesquels il y a beaucoup de groupes
  ab\'eliens d'ordre $n$.
\newblock {\em Annales de l'Institut Fourier}, 28(4):1--16, 1978.

\bibitem{paris}
R.~Paris.
\newblock An {Inequality} for the {Bessel} {Function} ${J}_{\nu}(\nu x)$.
\newblock {\em SIAM Journal on Mathematical Analysis}, 15(1):203--205, January
  1984.

\bibitem{rademacher}
Hans Rademacher.
\newblock On the {Partition} {Function} $p(n)$.
\newblock {\em Proceedings of the London Mathematical Society},
  s2-43(1):241--254, 1938.

\bibitem{watson}
G.~N. Watson.
\newblock {\em A {T}reatise on the {T}heory of {B}essel {F}unctions}.
\newblock Cambridge University Press, Cambridge, England, 1922.

\end{thebibliography}
\end{document}